\documentclass[11pt,a4paper,twoside]{amsart}
\usepackage{amssymb,amsmath,amsthm,graphicx,color}
\theoremstyle{plain}
\newtheorem{theorem}{Theorem}[section]
\newtheorem{proposition}[theorem]{Proposition}

\newtheorem{lemma}[theorem]{Lemma}

\theoremstyle{remark}
\newtheorem{remark}[theorem]{Remark}

\allowdisplaybreaks[4]
\numberwithin{equation}{section}

\newcommand{\R}{\mathbb{R}}

\renewcommand{\Im}{\operatorname{Im}}
\renewcommand{\Re}{\operatorname{Re}}
\newcommand{\I}{\infty}

\def\({\left(}
\def\){\right)}
\def\<{\left\langle}
\def\>{\right\rangle}
\def\le{\leqslant}
\def\ge{\geqslant}

\def\l{\lambda}

\newcommand{\rre}{\mathbb{R}}
\newcommand{\pt}{\partial}

\begin{document}

\title[Quadratic 
nonlinear Klein-Gordon equation in 2D]
{Modified scattering for the quadratic\\ 
nonlinear Klein-Gordon equation\\
in two dimensions}
\author[S.Masaki and J.Segata]
{Satoshi Masaki and Jun-ichi Segata}

\address{Department systems innovation\\
Graduate school of Engineering Science\\
Osaka University\\
Toyonaka Osaka, 560-8531, Japan}
\email{masaki@sigmath.es.osaka-u.ac.jp}

\address{Mathematical Institute, Tohoku University\\
6-3, Aoba, Aramaki, Aoba-ku, Sendai 980-8578, Japan}
\email{segata@m.tohoku.ac.jp}





\begin{abstract}
In this paper, we consider the long time behavior of solution 
to the quadratic nonlinear Klein-Gordon equation 
(NLKG) in two 
space dimensions:
$(\Box+1)u=\lambda|u|u$, $t\in\rre$, $x\in\rre^{2}$,
where $\Box=\pt_{t}^{2}-\Delta$ is d'Alembertian. 
For a given asymptotic profile $u_{\mathrm{ap}}$, 
we construct a solution $u$ to (NLKG) 
which converges to $u_{\mathrm{ap}}$ as $t\to\infty$.
Here the asymptotic profile $u_{\mathrm{ap}}$ is given by the leading term of the solution 
to the linear Klein-Gordon equation with a logarithmic phase correction.
Construction of a suitable approximate 
solution is based on Fourier series expansion of the nonlinearity.
\end{abstract}

\maketitle

\section{Introduction}

We consider the final state problem 
for the quadratic nonlinear Klein-Gordon equation in two 
space dimensions:
\begin{equation}
\left\{
\begin{array}{l}
\displaystyle{(\Box+1)u
=\lambda|u|u
\qquad t\in\rre,\ x\in\rre^{2},}\\
\displaystyle{
u-u_{\mathrm{ap}}\to0\qquad\text{in}\ L^{2}\ \ \text{as}\ t\to+\infty,}
\end{array}
\right. \label{K}
\end{equation}
where $\Box=\pt_{t}^{2}-\Delta$ is d'Alembertian, 
$u:\rre\times\rre^{2}\to\rre$ is an unknown
function, $u_{\mathrm{ap}}:\rre\times\rre^{2}\to\rre$ is a given function, 
and $\lambda$ is a non-zero real constant. 

There are many known results on the 
scattering for the nonlinear Klein-Gordon equation
\begin{eqnarray}
(\Box+1)u=\lambda|u|^{p-1}u,\qquad t\in\rre,\ x\in\rre^{n},
\label{KL}
\end{eqnarray}
where $p>1$ and $\lambda\in\rre\backslash\{0\}$. 
Here, we focus on the results on scattering for small data.
For the scattering results for large data, 
see \cite{B1,B2,GV,Na} for instance. 
For the case $p_{0}(n)<p\le1+4/(n-2)$ with  
$p_{0}(n)=(n+2+\sqrt{n^{2}+12n+4})/(2n)$, 
small data scattering for (\ref{KL}) was 
studied by many authors, see \cite{St,P1,P2} for instance. 
As for the case $p\le p_{0}(n)$, 
Klainerman \cite{Kla} and Shatah \cite{Shatah} independently proved the global existence of 
a solution to the Klein-Gordon equation 
with the quadratic nonlinearity for $n=3$ 
by using the vector field approach and the normal form, 
respectively. By using the vector field approach, 
Georgiev and Lecente \cite{GL} obtained a point-wise decay estimates 
for solutions to the (\ref{KL}) for $p>1+2/n$ with $n=1,2,3$. 
Hayashi and Naumkin have shown in \cite{HN4} that the nonlinear interaction in (\ref{KL}) 
is a \emph{short range type} for $p>1+2/n$ with $n=1,2$, i.e.,
solutions to (\ref{KL}) scatter to the solution to the linear Klein-Gordon equation if $p>1+2/n$. 
On the other hand, Glassey \cite{G} and Matsumura \cite{Ma} proved that 
the nonlinear interaction in (\ref{KL}) is a \emph{long range 
type} for $1<p\le1+2/n$ and $n\ge2$, namely, solutions to (\ref{KL}) 
do not scatter to the solution to the linear Klein-Gordon equation if $1<p\le1+2/n$.
A similar result was obtained in 
the case where $p=3$ and $n=1$ by Georgiev and Yordanov \cite{GY}. 
With these results, we see that the exponent $p=1+2/n$ is the borderline between the 
short range and long range scattering theories. 

We briefly explain 
why the exponent $p=1+2/n$ appears as the borderline. 
Roughly speaking, the nonlinearity is short range if and only 
if the $L^2$ norm of the nonlinear term is integrable in 
$t\in[1,\infty)$. Since the point-wise decay of the solution 
to the linear Klein-Gordon equation is $O(t^{-n/2})$ 
as $t\to\infty$, the $L^2$ norm of the nonlinear term $|u|^{p-1}u$ 
has the rate $O(t^{-n(p-1)/2})$.
Then, we observe that the nonlinearity is short range if and only if
the integral $\int_1^{\infty}t^{-n(p-1)/2}dt$ is finite.
The condition is nothing but $p>1+2/n$.
The argument also suggest that  
solutions of the nonlinear equation (\ref{KL}) with $p\le1+2/n$ 
may have an asymptotic behavior different from a solution 
of the linear Klein-Gordon equation. 
Thus, the threshold is $p=1+2/n$.

For the Klein-Gordon equation with the cubic 
nonlinearity 
in one dimension,  Georgiev and Yordanov \cite{GY} studied the 
point-wise decay of a solution to the initial value problem. 
Delort \cite{D} obtained 
an asymptotic profile of a global solution to the equation 
corresponding to the small initial data with compact support 
(see also Lindblad and Soffer \cite{LS2} for an alternative 
proof). 
The compact support assumption in \cite{D} is later
removed by Hayashi and Naumkin in \cite{HN5}. 
We also note that global existence and 
the asymptotic behavior of a solution to the Klein-Gordon 
equation with the cubic quasi-linear nonlinearity is studied by 
Moriyama \cite{M}, Katayama \cite{Ka}, and Sunagawa \cite{S} in one space dimension.
Concerning the Klein-Gordon equation 
with the quadratic nonlinearity in 
two dimensions, Ozawa, Tsutaya, and Tsutsumi 
\cite{OTT} proved a global existence result 
and characterized the asymptotic behavior of a small solution 
to (\ref{KL}) with a smooth, quadratic, 
semi-linear nonlinearity, i.e., nonlinear term 
depends on $u,\pt_{t}u,\nabla u$. 
Delort, Fang, and Xue \cite{DFX} extended Ozawa-Tsutaya-Tsutsumi's result to the case where the nonlinear term is quasi-linear. 
See also Kawahara and Sunagawa \cite{KS} and Katayama, Ozawa and Sunagawa 
\cite{KOS} for related works.

In this paper we consider the scattering problem for (\ref{KL}) with 
the critical nonlinearity $|u|u$ in two space dimensions. Especially, 
we consider the final state problem: For a given 
asymptotic profile $u_{\mathrm{ap}}$, we construct a solution to (\ref{K}) 
which converges to the given asymptotic profile as $t\to\infty$. 
Notice that the critical nonlinearity $|u|u$ in two space dimensions 
was out of scope in the previous works due to the lack of smoothness 
of the nonlinear term.
 
Let us introduce the asymptotic profile $u_{\mathrm{ap}}$ which we work with.
To this end, we first recall that 
the leading term of a solution to the linear Klein-Gordon equation 
\begin{eqnarray*}
\left\{
\begin{array}{l}
\displaystyle{(\Box+1)v
=0
\qquad t\in\rre,\ x\in\rre^{2},}\\
\displaystyle{
v(0,x)=\phi_{0}(x),\quad \pt_{t}v(0,x)=\phi_{1}(x)\qquad x\in\rre^{2}}
\end{array}
\right. 
\end{eqnarray*}
is given by 
\begin{eqnarray*}
t^{-1}{\bf 1}_{\{|x|<t\}}(t,x)P_{1}(\mu)
\cos\left(\langle\mu\rangle^{-1}t\right)+
t^{-1}{\bf 1}_{\{|x|<t\}}(t,x)Q_{1}(\mu)
\sin\left(\langle\mu\rangle^{-1}t\right),
\end{eqnarray*}
where $\mu=x/\sqrt{t^{2}-|x|^{2}}$, ${\bf 1}_{\Omega}(t,x)$ 
is the characteristic function supported on
$\Omega\subset\rre^{1+2}$, 
and 
\begin{eqnarray}
P_{1}(\mu)
&=&-\langle\mu\rangle^{2}\Im\hat{\phi}_{0}(\mu)
-\langle\mu\rangle \Re\hat{\phi}_{1}(\mu),\label{p1}\\
Q_{1}(\mu)
&=&\langle\mu\rangle^{2}\Re\hat{\phi}_{0}(\mu)
-\langle\mu\rangle \Im\hat{\phi}_{1}(\mu),\label{q1}
\end{eqnarray}
see \cite{H} for instance.
For given final state $(\phi_{0},\phi_{1})$, we define the asymptotic 
profile $u_{\mathrm{ap}}$ by
\begin{eqnarray}
u_{\mathrm{ap}}(t,x)&:=&
t^{-1}{\bf 1}_{\{|x|<t\}}(t,x)P_{1}(\mu)
\cos\left(\langle\mu\rangle^{-1}t
+\Psi(\mu)\log t\right)\label{up}\\
& &+
t^{-1}{\bf 1}_{\{|x|<t\}}Q_{1}(\mu)
\sin\left(\langle\mu\rangle^{-1}t+\Psi(\mu)\log t\right),
\nonumber
\end{eqnarray}
where the phase correction term is given by
\begin{eqnarray}
\Psi(\mu)=-\frac{4}{3\pi}\langle\mu\rangle
\left|\hat{\phi}_{0}(\mu)+i\langle\mu\rangle^{-1}\hat{\phi}_{1}(\mu)\right|.
\label{phase}
\end{eqnarray}
The final state $(\phi_0,\phi_1)$ is taken from the function space $Y$ defined by 
\begin{eqnarray*}
Y&:=&\{(\phi_{0},\phi_{1})\in{{\mathcal S}}'(\rre)\times{{\mathcal S}}'(\rre);
\|(\phi_{0},\phi_{1})\|_{Y}<\infty\},\\
\|(\phi_{0},\phi_{1})\|_{Y}
&:=&\|\phi_{0}\|_{H_{x}^{2}}+\|x\phi_{0}\|_{H_{x}^{3}}+\|x^{2}\phi_{0}\|_{H_{x}^{4}}\\
& &+\|\phi_{1}\|_{H_{x}^{1}}+\|x\phi_{1}\|_{H_{x}^{2}}+\|x^{2}\phi_{1}\|_{H_{x}^{3}}.
\end{eqnarray*} 

The main result in this paper is as follows. 

\begin{theorem} \label{thm:main} 
Let $(\phi_{0},\phi_{1})\in Y$. 
For $1/2<d<1$,
there exist a sufficiently large number $T\ge e$ and a sufficiently small number 
$\varepsilon>0$ such that 
if $\|(\phi_{0},\phi_{1})\|_{Y}<\varepsilon$
then there exists a unique solution $u(t)$ for
the equation (\ref{K}) satisfying
\begin{gather}
u \in C([T,\infty);L^2_x), \notag \\
\sup_{t \ge T} t^d \|u-u_{\mathrm{ap}}\|_{L^{\infty}((t,\infty);L_x^{2})}
<\infty,
\label{eq:conv}
\end{gather}
where the asymptotic profile $u_{\mathrm{ap}}$ 
is defined by (\ref{up}).
\end{theorem}

\begin{remark}
Concerning the final state problem for (\ref{KL}) with the cubic nonlinearity 
in one space dimension, 
Hayashi and Naumkin \cite{HN2} constructed 
a modified wave operators for (\ref{KL}) for small final data.
Furthermore, Lindblad and Soffer \cite{LS1} showed existence of a 
modified wave operators for (\ref{KL}) for
large final data in the case where $\lambda<0$. 
\end{remark}

\begin{remark}
The same result holds true for equations with a general quadratic nonlinearity $F(u):\R \to \R$ satisfying
$F(\lambda u)=\l^2 F(u)$ for all $\lambda>0$ and $u\in\R$.
See Remark \ref{rem:general} below for the detail.
\end{remark}

The rest of the paper is organized as follows. 
Section 2 is devoted to the exhibition of an outline for 
the proof of Theorem \ref{thm:main}.
The proof of Theorem \ref{thm:main} is based 
on the contraction principle via the integral 
equation of Yang-Feldman type associated with 
(\ref{K}) around a suitable 
approximate solution. 
The crucial part of the proof
is construction of the suitable approximate solution.
We summarize how to do this in this section.
%
%
%
In Section 3, we solve an abstract final 
value problem around an approximate solution.
Then, in Section 4, we show 
the approximate solution given in Section 2
satisfies the assumptions of 
the final value problem in Section 3,
and we completes the proof of Theorem~\ref{thm:main}.


\section{Outline of the proof of Theorem \ref{thm:main}
} 

In this section, we give an outline of the proof of Theorem~\ref{thm:main}.
For $T>0$, we define the function spaces $X_{T}$ by 
\begin{eqnarray*}
X_T&:=&\{ w \in C([T,\infty);L^2_x); \ \|w\|_{X_T} < \infty \}, \\
\|w\|_{X_T}&: =& \sup_{t \ge T} t^d (\|w\|_{L_{t}^{\infty}((t,\infty);H^{1/2}_x)}
+\|w\|_{L^4((t,\infty);L_{x}^{4})}),
\end{eqnarray*}
where $1/2<d<1$.
Put $N(u) = \lambda|u|u$. 
Let $A$ be a function satisfying
\begin{gather}
\|A(t)\|_{L^{\infty}_x} \le \eta t^{-1},
\label{eq:a1}\\
\begin{split}
\|(\Box+1)A(t)-N(A)(t)\|_{L^2_x}
\le \eta t^{-1-d}.
\label{eq:a2}
\end{split}
\end{gather}
We will prove in Section 3 that, once we find such a function $A$,
there exists 
a unique solution $u$
to the equation (\ref{K}) satisfying $u-A \in X_T$.
To prove this assertion,  
we employ the Strichartz estimate (Lemma \ref{lem:St})
and the contraction argument. 
Hence, it suffices to construct a function $A$ satisfying the conditions
\eqref{eq:a1} and \eqref{eq:a2}
for a given final state $(\phi_{0},\phi_{1})\in Y$. 
It will turn out that $A=u_{\mathrm{ap}}$ does not work well, and so that we need further modification.

We now explain how to construct the function $A=A(t,x)$ satisfying 
the conditions \eqref{eq:a1} and \eqref{eq:a2}. 
The conclusion is that the choice $A:=u_{\mathrm{ap}}+v_{\mathrm{ap}}$ works, where 
$u_{\mathrm{ap}}$ is the \emph{first approximation} given by (\ref{up})  
and $v_{\mathrm{ap}}$ is
\emph{the second approximation} which is of the form 
\begin{eqnarray}
v_{\mathrm{ap}}&:=&
t^{-2}{\bf 1}_{\{|x|<t\}}\sum_{n=2}^{\infty}P_{n}(\mu)
\cos\left(n\langle\mu\rangle^{-1}t
+n\Psi(\mu)\log t\right)\label{vp}\\
& &+
t^{-2}{\bf 1}_{\{|x|<t\}}\sum_{n=2}^{\infty}Q_{n}(\mu)
\sin\left(n\langle\mu\rangle^{-1}t+n\Psi(\mu)\log t\right).
\nonumber
\end{eqnarray}
Here the phase function $\Psi$ is the same as (\ref{phase}), and 
choice of $P_{n}$ and $Q_{n}$ will be specified later. 
Remark that $v_{\mathrm{ap}}(t)= O(t^{-1})$ in $L_{x}^{2}$.
Toward the conclusion, we will observe
(i) why the second approximation $v_{\mathrm{ap}}$ is required,
and (ii) what is the appropriate choice of $P_n$ and $Q_n$, by a somewhat heuristic argument.
Hereafter, we consider the case $|x|<t$ only because $u_{\mathrm{ap}}$ and $v_{\mathrm{ap}}$ 
are identically zero in the region $|x|\ge t$.

We first focus 
on the nonlinear part $N(u_{\mathrm{ap}})=\lambda|u_{\mathrm{ap}}|u_{\mathrm{ap}}$. 
Since $N(u)=\lambda|u|u$ is not polynomial in $(u,\overline{u})$, 
it becomes difficult to pick up a \emph{resonant part} from $N(u_{\mathrm{ap}})$. 
Taking a hint from our previous paper \cite{MS2}, 
we use the Fourier series expansion of $N(u_{\mathrm{ap}})$ to 
decompose $N(u_{\mathrm{ap}})$ into the resonant part and the rest, the \emph{non-resonant part}.
%
This decomposition is done as follows.
We rewrite $u_{\mathrm{ap}}$ as 
\begin{eqnarray*}
u_{\mathrm{ap}}=\left\{\begin{aligned}
&t^{-1}{\bf 1}_{\{|x|<t\}}\sqrt{P_{1}(\mu)^{2}+Q_{1}(\mu)^{2}}
\cos(\alpha-\beta)\quad
&& \text{if}\ P_{1}(\mu)^{2}+Q_{1}(\mu)^{2}\neq0,\\
&0
&& \text{if}\ P_{1}(\mu)^{2}+Q_{1}(\mu)^{2}=0,
\end{aligned}
\right.
\end{eqnarray*}
where $\alpha=\langle\mu\rangle^{-1}t+\Psi(\mu)\log t$ 
and $\beta\in(0,2\pi]$ is given by 
\begin{eqnarray}
\cos\beta=\frac{P_{1}(\mu)}{\sqrt{P_{1}(\mu)^{2}+Q_{1}(\mu)^{2}}},
\quad
\sin\beta=\frac{Q_{1}(\mu)}{\sqrt{P_{1}(\mu)^{2}+Q_{1}(\mu)^{2}}}.
\label{y0}
\end{eqnarray}
Then, we have
\begin{eqnarray}
\label{y01}\\
N(u_{\mathrm{ap}})
&=&
\lambda t^{-2}(P_{1}(\mu)^{2}+Q_{1}(\mu)^{2})c_1\cos(\alpha-\beta)
\nonumber\\
& &+
\lambda t^{-2}(P_{1}(\mu)^{2}+Q_{1}(\mu)^{2})\sum_{n\ge2}c_{n}\cos(n\alpha-n\beta)
\nonumber\\
&=&
\frac{8}{3\pi}\lambda t^{-2}\sqrt{P_{1}(\mu)^{2}+Q_{1}(\mu)^{2}}
(P_{1}(\mu)\cos\alpha+Q_{1}(\mu)\sin\alpha)
\nonumber\\
& &+\lambda t^{-2}(P_{1}(\mu)^{2}+Q_{1}(\mu)^{2})
\sum_{n\ge2}
c_{n}
\cos(n\alpha-n\beta)
\nonumber\\
&=:&N_{\mathrm{r}}(u_{\mathrm{ap}})+N_{\mathrm{nr}}(u_{\mathrm{ap}}),
\nonumber
\end{eqnarray}
where $c_{n}$ are the Fourier coefficients for the function $|\cos\theta|\cos\theta$:
\begin{eqnarray*}
c_{n}=\frac{1}{\pi}\int_{0}^{2\pi}|\cos\theta|\cos\theta\cos n\theta d\theta
=
\left\{\begin{aligned}
&-\frac{8}{\pi}\frac{\sin(\frac{n}{2}\pi)}{n(n^{2}-4)}\quad
&& \text{if}\ n\ \text{is odd},\\
&0
&& \text{if}\ n\ \text{is even}.
\end{aligned}
\right.
\end{eqnarray*}
This kind of technique was also used in Sunagawa \cite{S} to pick up 
the resonant term from the cubic nonlinearity in one space dimension. 
In that case the Fourier series for $N(u_{\mathrm{ap}})$  consists of four terms. 
We would emphasize that, in our setting, the Fourier series consists of \emph{infinitely many terms}, so we need to 
take care of the convergence of the Fourier series,
which seems a new ingredient.
Fortunately, it will turn out that the nonlinearity $|u|u$ has enough smoothness to ensure the convergence of 
the Fourier series for $|u|u$. 
We mention similar but slightly different expansion of a nonlinearity into a infinite Fourier sires is used
by the first author and Miyazaki \cite{MMy} in the context of nonlinear Schr\"odinger equation.

Since both of the resonant and non-resonant parts are 
$O(t^{-1})$ in $L_{x}^{2}$, we need to 
cancel out those terms by the linear part, otherwise \eqref{eq:a2} fails.
Thanks to the phase correction $\Psi$, we have the desired cancellation of the resonant part.
Namely, we have
\begin{eqnarray}
(\Box+1)u_{\mathrm{ap}}
=N_{\mathrm{r}}(u_{\mathrm{ap}})+O(t^{-2}(\log t)^{2}),\qquad in\ L^{2}
\nonumber
\end{eqnarray}
as $t\to\infty$, 
see Lemma \ref{lem:pf1} for the detail. 
We then add a \emph{second approximation} 
$v_{\mathrm{ap}}$ of $u$, given in \eqref{vp},
in order to cancel the non-resonant term $N_{\mathrm{nr}}(u_{\mathrm{ap}})$ out.
This is the reason why we need the second approximation $v_{\mathrm{ap}}$.

To obtain the desired cancellation, we will choose suitable $P_{n}$ and $Q_{n}$.
More precisely, we choose them so that the leading term of 
$n$-th term of $(\Box+1)v_{\mathrm{ap}}$ and $n$-th term of the Fourier 
expansion of $N_{\mathrm{nr}}(u_{\mathrm{ap}})$ coincide. By a computation, we have
\begin{eqnarray*}
(\Box+1)v_{\mathrm{ap}}&=&
t^{-2}\sum_{n=2}^{\infty}(1-n^{2})P_{n}(\mu)
\cos\left(n\langle\mu\rangle^{-1}t
+n\Psi(\mu)\log t\right)\\
& &+
t^{-2}\sum_{n=2}^{\infty}(1-n^{2})Q_{n}(\mu)
\sin\left(n\langle\mu\rangle^{-1}t+n\Psi(\mu)\log t\right),\\
& &+O(t^{-2}(\log t)^{2}),\qquad in\ L^{2}
\end{eqnarray*}
as $t\to\infty$, see Lemma \ref{lem:vn} for the detail.
Hence, we obtain the specific choice
\begin{eqnarray}
\label{pn}\\
P_{n}(\mu)
=
\left\{\begin{aligned}
&\frac{8\sin(\frac{n}{2}\pi)}{\pi n(n^{2}-1)(n^{2}-4)}
\lambda(P_{1}(\mu)^{2}+Q_{1}(\mu)^{2})\cos(n\beta)
&& \text{if}\ n\ \text{is odd},\\
&0
&& \text{if}\ n\ \text{is even},
\end{aligned}
\right.\nonumber
\end{eqnarray}
\begin{eqnarray}
\label{qn}\\
Q_{n}(\mu)
=\left\{\begin{aligned}
&\frac{8\sin(\frac{n}{2}\pi)}{\pi n(n^{2}-1)(n^{2}-4)}
\lambda(P_{1}(\mu)^{2}+Q_{1}(\mu)^{2})\sin(n\beta)
&& \text{if}\ n\ \text{is odd},\\
&0
&& \text{if}\ n\ \text{is even}.
\end{aligned}
\right.\nonumber
\end{eqnarray}
With this choice, the leading term of 
the $n$-th term of $(\Box+1)v_{\mathrm{ap}}$ and the $n$-th term of the Fourier 
expansion for $N_{\mathrm{nr}}(u_{\mathrm{ap}})$ successfully cancel out each other. 
Further, it will turn out in Section 4 that the error term can be handled thanks to fast 
decay of $P_n$ and $Q_n$ in $n$.
Remark that the coefficients of $P_n$ and $Q_n$ are order $O(|n|^{-5})$ as $|n|\to\I$.
The decay rate of the Fourier coefficients reflects the smoothness of the nonlinearity $\l |u|u$. 
Thus, we see that $A=u_{\mathrm{ap}}+v_{\mathrm{ap}}$ 
satisfies the conditions \eqref{eq:a1} and \eqref{eq:a2}.

This kind of approximation was introduced in H\"{o}rmander \cite{H} 
for the Klein-Gordon equation with \emph{polynomial} nonlinearity in $(u,\overline{u})$. 
See also \cite{MTT,ST} for the nonlinear Schr\"{o}dinger equation 
with polynomial nonlinearity in $(u,\overline{u})$. 
 
\begin{remark}\label{rem:general}
Let us consider a generalization of Theorem \ref{thm:main}. 
Notice that any real-valued quadratic nonlinearity  
can be expressed as the linear combination of $|u|u$ and $u^2$. Indeed, if 
$F(u):\R \to\R$ satisfies $F(\lambda u)=\lambda^2 F(u)$ for any $\lambda>0$ and 
$u\in{{\mathbb R}}$, then we see 
\[
	F(u) = \frac{F(1)+F(-1)}{2} u^2 + \frac{F(1)-F(-1)}{2} |u|u.
\]
As for the even nonlinearity $u^2$, it is easy to pick up the resonant/non-resonant part from $u_{\mathrm{ap}}^{2}$ 
because the Fourier series of $u_{\mathrm{ap}}^{2}$ consists of the zeroth and the second terms only.
In particular, it contains no resonant part and so
existence of the even part does not change (the leading term of) the asymptotic profile. 
Thus, we can generalize Theorem \ref{thm:main} 
for general real-valued quadratic nonlinearity in two dimensions. 
Note that the final state problem for the Klein-Gordon equation 
with the nonlinearity $u^{2}$ in the two dimensions was studied by 
Hayashi and Naumkin \cite{HN3} by using the normal form method.  
It is an interesting problem to generalize Theorem \ref{thm:main} 
to equations with \emph{complex-valued} quadratic nonlinearities. 
As a related problem, we mention that Sunagawa \cite{Su} obtained the point-wise decay 
estimate of the complex valued solution to the initial value problem for 
the one dimensional cubic nonlinear Klein-Gordon equation. 
\end{remark}


\section{The final value problem} 

In this section, we solve a Cauchy problem at infinite initial time for the equation (\ref{K})
in an abstract framework.
Let $A(t,x)$ be a given asymptotic profile of a solution to (\ref{K}).
We show that if $A(t,x)$ is well-chosen then we obtain a solution which asymptotically behaves like $A(t,x)$.
The main respect is to make it clear how nicely an asymptotic profile should be chosen.

Let $N(u)=\lambda|u|u$. We introduce the \emph{error function} $F(t,x)$ by
\begin{eqnarray}
F:=(\Box+1)A -N(A).
\label{eq:defR}
\end{eqnarray}
For $T>0$, we introduce the following function space
\begin{gather*}
X_T =\{ w \in C([T,\infty);L^2_x); \ \|w\|_{X_T} < \infty \},
\end{gather*}
where
\begin{equation*}
\|w\|_{X_T} = \sup_{t \ge T} t^d (\|w\|_{L^{\infty}((t,\infty);H^{1/2}_x)}
+\|w\|_{L^4((t,\infty);L^{4}_x)})
\end{equation*}
with $1/2<d<1$.
For $\rho >0$ and $T>0$, we define
\begin{gather*}
\widetilde{X}_T(\rho)= \{ w \in C([T,\infty);L^2_x);
\ \|w\|_{X_T} \le \rho \}.
\end{gather*}
The function space $X_T$ is a Banach space with the norm $\|\cdot\|_{X_T}$
and $\widetilde{X}_T(\rho)$ is a complete metric space with
the $\|\cdot\|_{X_T}$-metric.

\begin{proposition} \label{prop:FVP}
Let $d$ be a constant such that $1/2<d<1$. 
Then there exist a sufficiently large $T>0$ and 
a sufficiently small $\eta>0$ such that if 
$A(t,x)$ satisfies
\begin{gather}
\|A(t)\|_{L^{\infty}_x} \le \eta t^{-1},
\label{eq:A1}\\
\begin{split}
\|F(t)\|_{L^2_x}\le \eta t^{-1-d},
\label{eq:A2}
\end{split}
\end{gather}
where $F(t,x)$ is given by (\ref{eq:defR}),
then there exists a unique solution $u$
for the equation (\ref{K}) satisfying
\begin{eqnarray*}
u \in C([T,\infty);L^2_x), 
\end{eqnarray*}
\begin{eqnarray}
\sup_{t \ge T} t^d (\|u-A\|_{L^{\infty}((t,\infty);H^{1/2}_x)}
+\|u-A\|_{L^4((t,\infty);L^{4}_x)}) <\infty.
\label{pr}
\end{eqnarray}
\end{proposition}

To prove Proposition \ref{prop:FVP}, we use the following 
inhomogeneous Strichartz estimates associated with the 
Klein-Gordon equation. Let 
\begin{eqnarray}
{{\mathcal G}}[g](t)
:=\int_t^{\infty} \sin((t-\tau)\sqrt{1-\Delta})(1-\Delta)^{-1/2}g(\tau)d\tau.
\label{g}
\end{eqnarray}

\begin{lemma}\label{lem:St} Let $2\le q<\infty$ and $1/p+1/q=1/2$. Then we have
\begin{eqnarray*}
\|{{\mathcal G}}[g]
\|_{L_{t}^{p}([T,\infty),L_{x}^{q})}
&\le&C\|(1-\Delta)^{1/2-2/q}g\|_{L_{t}^{p'}([T,\infty),L_{x}^{q'})},\\
\|{{\mathcal G}}[g]
\|_{L_{t}^{\infty}([T,\infty),L_{x}^{2})}
&\le&C\|(1-\Delta)^{-1/q}g\|_{L_{t}^{p'}([T,\infty),L_{x}^{q'})},\\
\|{{\mathcal G}}[g]\|_{L_{t}^{p}([T,\infty),L_{x}^{q})}
&\le&C\|(1-\Delta)^{-1/q}g\|_{L_{t}^{1}([T,\infty),L_{x}^{2})}.
\end{eqnarray*}
\end{lemma}

\begin{proof}[Proof of Lemma \ref{lem:St}]
The above inequalities follow from the $L^{p}$-$L^{q}$ estimate 
for the solution to the Klein-Gordon equation by \cite{MSW} and the duality 
argument by \cite{Y}. Since the proof is now standard, we omit the detail.
\end{proof}

\begin{proof}[Proof of Proposition \ref{prop:FVP}]
We put $v=u-A$.
Then the equation (\ref{K}) is equivalent to
\begin{equation}
(\Box+1)v = N(v+A) -N(A) -F,
\label{K'}
\end{equation}
where $F$ is defined by \eqref{eq:defR}.
The associate integral equation to the equation (\ref{K'}) is
\begin{eqnarray}
v=
{{\mathcal G}}[\{ N(v+A) -N(A)\} -F],
\label{eq:INT}
\end{eqnarray}
where ${{\mathcal G}}$ is given by (\ref{g}). It suffices to show the existence of a unique solution $v$
to the equation \eqref{eq:INT} in $X_T$ for sufficiently large
$T>0$ and sufficiently small $\eta>0$.
We prove this assertion by the contraction argument.
Define the nonlinear operator $\Phi$ by
\begin{eqnarray*}
\Phi v:=
{{\mathcal G}}[\{ N(v+A) -N(A)\} -F]
\end{eqnarray*}
for $v \in \widetilde{X}_T(\rho)$. We show that for any $\rho >0$,
$\Phi$ is a contraction map on $\widetilde{X}_T(\rho)$
if $T>0$ is sufficiently large and $\eta>0$ is sufficiently small.
Let $\rho >0$ be arbitrary,
and $T,\eta>0$ which will be determined below.
Let $v \in \widetilde{X}_T(\rho)$ and $t \ge T$.
By the assumptions and Lemma \ref{lem:St}, we see
\begin{equation*}
\begin{split}
\|(\Phi &v)(t)\|_{L^{\infty}((t,\infty);H^{1/2}_x)}
+\|\Phi v\|_{L^4((t,\infty);L^{4}_x)} \\
\le & C(\|v^{2}\|_{L^{4/3}((t,\infty);L^{4/3}_x)}
+\|(1-\Delta)^{-1/4}Av\|_{L^1((t,\infty);L^2_x)} \\
& +\|(1-\Delta)^{-1/4}F\|_{L^1((t,\infty);L^2_x)}) \\
\le & C \biggl\{ \|v\|_{L^4 ((t,\infty);L^{4}_x)}
\left( \int_t^{\infty} \|v(\tau)\|_{L^2_x}^2 \,d\tau \right)^{1/2} 
+\int_t^{\infty} \|A(\tau)\|_{L^{\infty}_x}
\|v(\tau)\|_{L^2_x}\,d\tau \\
& +\int_t^{\infty}\|F(\tau)\|_{L^2_x}\,d\tau \biggr\} \\
\le & C \biggl\{ \rho t^{-d}
\left( \int_t^{\infty} \rho^2\tau^{-2d} \,d\tau \right)^{1/2}
+\int_t^{\infty} \eta \tau^{-1} 
\rho \tau^{-d} \,d\tau +\int_t^{\infty} \eta \tau^{-1-d} \,d\tau \biggr\} \\
\le & C t^{-d} (\rho^2 t^{-d+1/2} +\rho \eta
+\eta ).
\end{split}
\end{equation*}
Therefore we obtain
\begin{equation}
\|\Phi v\|_{X_T} \le C_{1}(\rho^2 T^{-d+1/2} +\rho \eta
+\eta).
\label{eq:into}
\end{equation}
In the same way as above, for $v_1,v_2 \in \widetilde{X}_T(\rho)$,
we can show
\begin{equation}
\begin{split}
\|\Phi v_1 & -\Phi v_2\|_{X_T} \\
\le & C_{2}((\|v_1\|_{X_T} +\|v_2\|_{X_T}) T^{-d+1/2}
+\eta)\|v_1-v_2\|_{X_T} \\
\le & C_{2}(\rho T^{-d+1/2} +\eta) \|v_1-v_2\|_{X_T}.
\label{eq:cont}
\end{split}
\end{equation}
We note that for $\rho >0$, there exists
a sufficiently large $T>0$ and a sufficiently small $\eta >0$ 
such that
\begin{gather*}
C_{1}(\rho^2 T^{-d+1/2}
+\rho \eta +\eta)
\le \rho, \\
C_{2}(\rho T^{-d+1/2} +\eta)
\le \frac{1}{2},
\end{gather*}
since $d>1/2$.
From this observation, the estimates \eqref{eq:into} and \eqref{eq:cont} show
 that the operator $\Phi$ is a contraction map
on $\widetilde{X}_T(\rho)$ for sufficiently large $T>0$ 
and sufficiently small $\eta>0$.
Therefore for any $\rho >0$,
there exist $T>0$, $\eta>0$, and a unique solution to
the integral equation \eqref{eq:INT} in $\widetilde{X}_T(\rho)$.
The uniqueness of solutions to the equation
\eqref{eq:INT} in $X_T$ follows from
the first inequality of the estimate \eqref{eq:cont} for
solutions $v_1 \in X_T$ and $v_2 \in X_T$.
Hence 
the equation \eqref{eq:INT} has a unique solution in $X_T$.
This completes the proof of Proposition \ref{prop:FVP}.
\end{proof}


\section{Construction of suitable asymptotic profile} \label{sec:prf}

In this section, we complete the proof of Theorem~\ref{thm:main} 
by showing that the asymptotic profile $A(t,x)$ 
introduced in Section 2 
satisfies the assumptions in Proposition~\ref{prop:FVP}. 

\begin{proposition} \label{prop:ua} 
Assume that the final state  
$(\phi_{0},\phi_{1})$ satisfies $(\phi_{0},\phi_{1})\in Y$. 
Let $u_{\mathrm{ap}}$ be defined by (\ref{up}), where 
$P_{1}, Q_{1}$ and $\Psi$ are given by (\ref{p1}), (\ref{q1}) and (\ref{phase}), 
respectively. Let $v_{\mathrm{ap}}$ be defined by (\ref{vp}), where $P_{n}$ and 
$Q_{n}$ are given by (\ref{pn}) and (\ref{qn}). 
Then for $A=u_{\mathrm{ap}}+v_{\mathrm{ap}}$, there exist positive constants $C$ 
such that the inequalities 
\begin{eqnarray}
\|A(t)\|_{L^{\infty}_x} &\le& Ct^{-1}
\|(\phi_{0},\phi_{1})\|_{Y}(1+\|(\phi_{0},\phi_{1})\|_{Y}),
\label{eq:ua1}\\
\lefteqn{\|(\Box+1)A(t) -N(A(t))\|_{L^2_x}\qquad\qquad}\qquad\qquad\label{eq:ua2}\\
&\le&Ct^{-2}(\log t)^{2}\|(\phi_{0},\phi_{1})\|_{Y}(1+\|(\phi_{0},\phi_{1})\|_{Y}^{3})
\nonumber
\end{eqnarray}
hold for any $(\phi_{0},\phi_{1})\in Y$ and $t\ge e$.
\end{proposition}
To prove Proposition \ref{prop:ua}, 
we first calculate $(\Box+1)u_{\mathrm{ap}}$.

\begin{lemma}\label{lem:pf1} 
Assume that $(\phi_{0},\phi_{1})\in Y$.
Let $u_{\mathrm{ap}}$ be defined by (\ref{up}), where $P_{1}, Q_{1}$ and $\Psi$ are 
given by (\ref{p1}), (\ref{q1}) and (\ref{phase}), respectively. 
Let $N_{\mathrm{r}}(u_{\mathrm{ap}})$ be given by (\ref{y01}). Then it holds that
\begin{eqnarray}
\lefteqn{\|(\Box+1)u_{\mathrm{ap}}-N_{\mathrm{r}}(u_{\mathrm{ap}})\|_{L_{x}^{2}}
}\label{N2}\\
&\le&Ct^{-2}(\log t)^{2}\|(\phi_{0},\phi_{1})\|_{Y}
(1+\|(\phi_{0},\phi_{1})\|_{Y}^{2}).
\nonumber
\end{eqnarray}
\end{lemma}
\begin{proof}[Proof of Lemma \ref{lem:pf1}]
A simple calculation 
shows
\begin{eqnarray}
\lefteqn{\qquad(\Box+1)\{t^{-m}\cos\left(n\langle\mu\rangle^{-1}t\right)\}}
\label{ab1}\\
&=&(1-n^{2})t^{-m}\cos\left(n\langle\mu\rangle^{-1}t\right)
+2n(m-1)t^{-m-1}\langle\mu\rangle\sin\left(n\langle\mu\rangle^{-1}t\right)
\nonumber\\
& &+m(m+1)t^{-m-2}\cos\left(n\langle\mu\rangle^{-1}t\right)
\nonumber
\end{eqnarray}
for $m,n\in{{\mathbb Z}}_{+}$. In a similar way, 
\begin{eqnarray*}
\lefteqn{(\Box+1)\{t^{-m}\sin\left(n\langle\mu\rangle^{-1}t\right)\}}\\
&=&(1-n^{2})t^{-m}\sin\left(n\langle\mu\rangle^{-1}t\right)
-2n(m-1)t^{-m-1}\langle\mu\rangle\cos\left(n\langle\mu\rangle^{-1}t\right)
\\
& &+m(m+1)t^{-m-2}\sin\left(n\langle\mu\rangle^{-1}t\right).
\end{eqnarray*}
We now consider a function of the form $g=g(t,\mu)$.
Regarding $(s,\mu)=(t,x/\sqrt{t^2-|x|^2})$ as new variables, one obtains
\begin{eqnarray*}
\pt_{t}g
&=&\pt_{s}g-s^{-1}\langle \mu \rangle^{2}\mu_{1}\pt_{\mu_{1}}g
-s^{-1}\langle \mu \rangle^{2}\mu_{2}\pt_{\mu_{2}}g,\\
\pt_{x_{1}}g
&=&s^{-1}\langle \mu \rangle(1+\mu_{1}^{2})\pt_{\mu_{1}}g
+s^{-1}\langle \mu \rangle\mu_{1}\mu_{2}\pt_{\mu_{2}}g,\\
\pt_{x_{2}}g
&=&s^{-1}\langle \mu \rangle\mu_{1}\mu_{2}\pt_{\mu_{1}}g
+s^{-1}\langle \mu \rangle(1+\mu_{2}^{2})\pt_{\mu_{2}}g.
\end{eqnarray*}
Hence, 
\begin{eqnarray}
\Box (g(t,\mu))
&=&
\pt_{s}^{2}g
-s^{-2}\langle \mu \rangle^2(1+\mu_{1}^{2})
\pt_{\mu_{1}}^{2}g
-s^{-2}\langle \mu \rangle^2(1+\mu_{2}^{2})
\pt_{\mu_{2}}^{2}g
\label{ab2}\\
& &-2s^{-1}\langle \mu \rangle^2\mu_{1}\pt_{s}\pt_{\mu_{1}}g
-2s^{-1}\langle \mu \rangle^2\mu_{2}\pt_{s}\pt_{\mu_{2}}g
\nonumber\\
& &-2s^{-2}\langle \mu \rangle^2\mu_{1}\mu_{2}\pt_{\mu_{1}}\pt_{\mu_{2}}g
\nonumber\\
& &-2s^{-2}\langle \mu \rangle^2\mu_{1}\pt_{\mu_{1}}g
-2s^{-2}\langle \mu \rangle^2\mu_{2}\pt_{\mu_{2}}g.
\nonumber
\end{eqnarray}

By using the above identities, we now calculate 
$(\Box+1)u_{\mathrm{ap}}$. 
To this end, we split it into the following four pieces. 
\begin{eqnarray}
\lefteqn{(\Box+1)u_{\mathrm{ap}}}\label{Z1}\\
&=&(\Box+1)\left\{t^{-1}P_{1}(\mu)
\cos\left(\langle\mu\rangle^{-1}t\right)\cos\left(\Psi(\mu)\log t\right)\right\}
\nonumber\\
& &-(\Box+1)\left\{t^{-1}P_{1}(\mu)
\sin\left(\langle\mu\rangle^{-1}t\right)\sin\left(\Psi(\mu)\log t\right)\right\}
\nonumber\\
& &+(\Box+1)\left\{t^{-1}Q_{1}(\mu)
\sin\left(\langle\mu\rangle^{-1}t\right)\cos\left(\Psi(\mu)\log t\right)\right\}
\nonumber\\
& &+(\Box+1)\left\{t^{-1}Q_{1}(\mu)
\cos\left(\langle\mu\rangle^{-1}t\right)\sin\left(\Psi(\mu)\log t\right)\right\}
\nonumber\\
&=:&I_{1}+I_{2}+I_{3}+I_{4}.
\nonumber
\end{eqnarray}
Further, we split $I_{1}$ into the following five pieces:
\begin{eqnarray}
I_{1}&=&
(\Box+1)\left\{t^{-1}\cos\left(\langle\mu\rangle^{-1}t\right)\right\}
P_{1}(\mu)\cos\left(\Psi(\mu)\log t\right)\label{Z7}\\
& &+t^{-1}\cos\left(\langle\mu\rangle^{-1}t\right)
\Box\left\{P_{1}(\mu)\cos\left(\Psi(\mu)\log t\right)\right\}
\nonumber\\
& &+2\pt_{t}\left\{t^{-1}\cos\left(\langle\mu\rangle^{-1}t\right)\right\}
\pt_{t}\left\{P_{1}(\mu)\cos\left(\Psi(\mu)\log t\right)\right\}
\nonumber\\
& &-2\pt_{x_{1}}\left\{t^{-1}\cos\left(\langle\mu\rangle^{-1}t\right)\right\}
\pt_{x_{1}}\left\{P_{1}(\mu)\cos\left(\Psi(\mu)\log t\right)\right\}
\nonumber\\
& &-2\pt_{x_{2}}\left\{t^{-1}\cos\left(\langle\mu\rangle^{-1}t\right)\right\}
\pt_{x_{2}}\left\{P_{1}(\mu)\cos\left(\Psi(\mu)\log t\right)\right\}
\nonumber\\
&=:&J_{1}+J_{2}+J_{3}+J_{4}+J_{5}.
\nonumber
\end{eqnarray}
By (\ref{ab1}), we see 
\begin{eqnarray}
|J_{1}(\mu)|\le Ct^{-3}|P_{1}(\mu)|.\label{ab3}
\end{eqnarray}
By (\ref{ab2}), we have
\begin{eqnarray}
|J_{2}(\mu)|
&\le&Ct^{-3}(\log t)^{2}\{(|P_{1}(\mu)||\Psi(\mu)|^{2}+|P_{1}(\mu)||\Psi(\mu)|)
\label{ab4}\\
& &\qquad+\langle\mu\rangle^{3}
(|P_{1}(\mu)||\Psi(\mu)||D\Psi(\mu)|+|DP_{1}(\mu)||\Psi(\mu)|
\nonumber\\
& &\qquad\qquad\qquad+|P_{1}(\mu)||D\Psi(\mu)|+
|DP_{1}(\mu)|)
\nonumber\\
& &\qquad+\langle\mu\rangle^{4}
(|P_{1}(\mu)||D\Psi(\mu)|^{2}+|DP_{1}(\mu)||D\Psi(\mu)|
\nonumber\\
& &\qquad\qquad\qquad+|P_{1}(\mu)||D^{2}\Psi(\mu)|+
|D^{2}P_{1}(\mu)|)\},\nonumber
\end{eqnarray}
where $|Df(\mu)|=|\pt_{\mu_{1}}f(\mu)|+|\pt_{\mu_{2}}f(\mu)|$ and 
$|D^{2}f(\mu)|=|\pt_{\mu_{1}}^{2}f(\mu)|+|\pt_{\mu_{1}}\pt_{\mu_{2}}f(\mu)|
+|\pt_{\mu_{2}}^{2}f(\mu)|$. 
For $J_{j}$, $j=3,4,5$, an elementary calculation yields 
\begin{eqnarray*}
J_{3}
&=&-2t^{-1}\langle\mu\rangle
\sin\left(\langle\mu\rangle^{-1}t\right)\\
& &\qquad\times
\{\pt_{t}-t^{-1}\langle \mu \rangle^{2}\mu_{1}\pt_{\mu_{1}}
-t^{-1}\langle \mu \rangle^{2}\mu_{2}\pt_{\mu_{2}}\}
P_{1}(\mu)\cos\left(\Psi(\mu)\log t\right)\\
& &-2t^{-2}\cos\left(\langle\mu\rangle^{-1}t\right)\\
& &\qquad\times
\{\pt_{t}-t^{-1}\langle \mu \rangle^{2}\mu_{1}\pt_{\mu_{1}}
-t^{-1}\langle \mu \rangle^{2}\mu_{2}\pt_{\mu_{2}}\}
P_{1}(\mu)\cos\left(\Psi(\mu)\log t\right)\\
&=:&J_{31}+J_{32}+J_{33}+J_{34}+J_{35}+J_{36},\\
J_{4}
&=&-2t^{-1}\mu_{1}
\sin\left(\langle\mu\rangle^{-1}t\right)\\
& &\qquad\times
\{t^{-1}\langle \mu \rangle(1+\mu_{1}^{2})\pt_{\mu_{1}}
+t^{-1}\langle \mu \rangle\mu_{1}\mu_{2}\pt_{\mu_{2}}\}
P_{1}(\mu)\cos\left(\Psi(\mu)\log t\right)\\
&=:&J_{41}+J_{42},\\
J_{5}
&=&-2t^{-1}\mu_{2}
\sin\left(\langle\mu\rangle^{-1}t\right)\\
& &\qquad\times
\{t^{-1}\langle \mu \rangle\mu_{1}\mu_{2}\pt_{\mu_{1}}
+t^{-1}\langle \mu \rangle(1+\mu_{2}^{2})\pt_{\mu_{2}}\}
P_{1}(\mu)\cos\left(\Psi(\mu)\log t\right)\\
&=:&J_{51}+J_{52}.
\end{eqnarray*}
Since $J_{32}+J_{41}+J_{51}=0$ and $J_{33}+J_{42}+J_{52}=0$, we see that 
\begin{eqnarray}
\lefteqn{J_{3}+J_{4}+J_{5}}\label{ab5}\\
&=&J_{31}+J_{34}+J_{35}+J_{36}\nonumber\\
&=&2t^{-2}\langle\mu\rangle P_{1}(\mu)\Psi(\mu)
\sin\left(\langle\mu\rangle^{-1}t\right)\sin\left(\Psi(\mu)\log t\right)
\nonumber\\
& &+J_{34}+J_{35}+J_{36}. \nonumber
\nonumber
\end{eqnarray}
Substituting (\ref{ab5}) into (\ref{Z7}), we obtain
\begin{eqnarray}
\qquad\ \ I_{1}=2t^{-2}\langle\mu\rangle P_{1}(\mu)\Psi(\mu)
\sin\left(\langle\mu\rangle^{-1}t\right)\sin\left(\Psi(\mu)\log t\right)
+R_{11},\label{Z2}
\end{eqnarray}
where $R_{11}=J_{1}+J_{2}+J_{34}+J_{35}+J_{36}$. Hence 
by (\ref{ab3}) and (\ref{ab4}),  
\begin{eqnarray}
\lefteqn{|R_{11}(t,\mu)|}\label{r11}\\
&\le&
Ct^{-3}(\log t)^{2}\{(|P_{1}(\mu)||\Psi(\mu)|^{2}+|P_{1}(\mu)||\Psi(\mu)|
+|P_{1}(\mu)|)
\nonumber\\
& &\qquad+\langle\mu\rangle^{3}
(|P_{1}(\mu)||\Psi(\mu)||D\Psi(\mu)|+|DP_{1}(\mu)||\Psi(\mu)|
\nonumber\\
& &\qquad\qquad\qquad+|P_{1}(\mu)||D\Psi(\mu)|+
|DP_{1}(\mu)|)
\nonumber\\
& &\qquad+\langle\mu\rangle^{4}
(|P_{1}(\mu)||D\Psi(\mu)|^{2}+|DP_{1}(\mu)||D\Psi(\mu)|
\nonumber\\
& &\qquad\qquad\qquad+|P_{1}(\mu)||D^{2}\Psi(\mu)|+
|D^{2}P_{1}(\mu)|)\}.\nonumber
\end{eqnarray}
In a similar way, we have
\begin{eqnarray}
\qquad\ \  I_{2}&=&2t^{-2}\langle\mu\rangle P_{1}(\mu)\Psi(\mu)
\cos\left(\langle\mu\rangle^{-1}t\right)\cos\left(\Psi(\mu)\log t\right)
+R_{12},\label{Z3}\\
I_{3}&=&-2t^{-2}\langle\mu\rangle Q_{1}(\mu)\Psi(\mu)
\cos\left(\langle\mu\rangle^{-1}t\right)\sin\left(\Psi(\mu)\log t\right)
+R_{13},\label{Z4}\\
I_{4}&=&-2t^{-2}\langle\mu\rangle Q_{1}(\mu)\Psi(\mu)
\sin\left(\langle\mu\rangle^{-1}t\right)\cos\left(\Psi(\mu)\log t\right)
+R_{14},\label{Z5}
\end{eqnarray}
where $R_{12}$ satisfies (\ref{r11}), and 
$R_{13}$ and $R_{14}$ satisfy (\ref{r11}) with $Q_{1}$ 
instead of $P_{1}$. Substituting (\ref{Z2})-(\ref{Z5}) 
into (\ref{Z1}), 
we obtain 
\begin{eqnarray}
\label{Z6}\\
\lefteqn{\ |(\Box+1)u_{\mathrm{ap}}-N_{\mathrm{r}}(u_{\mathrm{ap}})|}
\nonumber\\
&\le&|R_{11}|+|R_{12}|+|R_{13}|+|R_{14}|
\nonumber\\
&\le&Ct^{-3}(\log t)^{2}
\sum_{Z=P,Q}\bigl\{(|Z_{1}(\mu)||\Psi(\mu)|^{2}+|Z_{1}(\mu)||\Psi(\mu)|+|Z_{1}(\mu)|)
\nonumber\\
& &\qquad\qquad\qquad\qquad+\langle\mu\rangle^{3}
(|Z_{1}(\mu)||\Psi(\mu)||D\Psi(\mu)|+|DZ_{1}(\mu)||\Psi(\mu)|
\nonumber\\
& &\qquad\qquad\qquad\qquad\qquad\quad+|Z_{1}(\mu)||D\Psi(\mu)|+
|DZ_{1}(\mu)|)
\nonumber\\
& &\qquad\qquad\qquad\qquad+\langle\mu\rangle^{4}
(|Z_{1}(\mu)||D\Psi(\mu)|^{2}+|DZ_{1}(\mu)||D\Psi(\mu)|
\nonumber\\
& &\qquad\qquad\qquad\qquad\qquad\quad+|Z_{1}(\mu)||D^{2}\Psi(\mu)|+
|D^{2}Z_{1}(\mu)|)\bigl\}.
\nonumber
\end{eqnarray}
By simple calculations, we see that  
\begin{eqnarray}
|Z_{1}(\mu)|
&\le&C\left(\langle\mu\rangle^{2}|\hat{\phi}_{0}(\mu)|
+\langle\mu\rangle|\hat{\phi}_{1}(\mu)|\right),
\label{p10}\\
|DZ_{1}(\mu)|
&\le&C\left(\langle\mu\rangle|\hat{\phi}_{0}(\mu)|
+\langle\mu\rangle^{2}|D\hat{\phi}_{0}(\mu)|\right)
\label{p11}\\
& &+C\left(|\hat{\phi}_{1}(\mu)|+\langle\mu\rangle|D\hat{\phi}_{1}(\mu)|\right),
\nonumber\\
|D^{2}Z_{1}(\mu)|
&\le&C\left(|\hat{\phi}_{0}(\mu)|+\langle\mu\rangle|D\hat{\phi}_{0}(\mu)|
+\langle\mu\rangle^{2}|D^{2}\hat{\phi}_{0}(\mu)|\right)
\label{p12}\\
& &+C\left(\langle\mu\rangle^{-1}|\hat{\phi}_{1}(\mu)|
+|D\hat{\phi}_{1}(\mu)|+\langle\mu\rangle|D^{2}\hat{\phi}_{1}(\mu)|\right)
\nonumber
\end{eqnarray}
for $Z=P,Q$, and
\begin{eqnarray}
\qquad|\Psi(\mu)|
&\le&C\left(\langle\mu\rangle|\hat{\phi}_{0}(\mu)|
+|\hat{\phi}_{1}(\mu)|\right),
\label{psi0}\\
|D\Psi(\mu)|
&\le&C\left(|\hat{\phi}_{0}(\mu)|
+\langle\mu\rangle|D\hat{\phi}_{0}(\mu)|\right)
\label{psi1}\\
& &+C\left(\langle\mu\rangle^{-1}|\hat{\phi}_{1}(\mu)|+|D\hat{\phi}_{1}(\mu)|
\right),
\nonumber\\
|D^{2}\Psi(\mu)|
&\le&C\left(\langle\mu\rangle^{-1}|\hat{\phi}_{0}(\mu)|+|D\hat{\phi}_{0}(\mu)|
+\langle\mu\rangle|D^{2}\hat{\phi}_{0}(\mu)|\right)
\label{psi2}\\
& &+C\left(\langle\mu\rangle^{-2}|\hat{\phi}_{1}(\mu)|
+\langle\mu\rangle^{-1}|D\hat{\phi}_{1}(\mu)|+|D^{2}\hat{\phi}_{1}(\mu)|
\right).
\nonumber
\end{eqnarray}
Plugging (\ref{p10})-(\ref{p12}) and (\ref{psi0})-(\ref{psi2}) 
into (\ref{Z6}), we have 
\begin{eqnarray*}
\lefteqn{|(\Box+1)u_{\mathrm{ap}}-N_{\mathrm{r}}(u_{\mathrm{ap}})|}\\
&\le&Ct^{-3}(\log t)^{2}\langle\mu\rangle^{2}\\
& &\times\biggl\{
\left(1+\langle\mu\rangle|\hat{\phi}_{0}(\mu)|
+\langle\mu\rangle^{2}|D\hat{\phi}_{0}(\mu)|\right)^{2}
\left(\langle\mu\rangle^{2}|\hat{\phi}_{0}(\mu)|
+\langle\mu\rangle^{3}|D\hat{\phi}_{0}(\mu)|\right)\\
& &\quad+\left(1+\langle\mu\rangle|\hat{\phi}_{0}(\mu)|
\right)\langle\mu\rangle^{4}|D^{2}\hat{\phi}_{0}(\mu)|\\
& &\quad+\left(1+|\hat{\phi}_{1}(\mu)|
+\langle\mu\rangle|D\hat{\phi}_{1}(\mu)|\right)^{2}
\left(\langle\mu\rangle|\hat{\phi}_{1}(\mu)|
+\langle\mu\rangle^{2}|D\hat{\phi}_{1}(\mu)|\right)\\
& &\quad+\left(1+|\hat{\phi}_{1}(\mu)|
\right)\langle\mu\rangle^{3}|D^{2}\hat{\phi}_{1}(\mu)|\biggl\}.
\end{eqnarray*}
Therefore, taking $L^{2}$ norm for the above inequality with respect to $x$ variable, 
we have the estimate \eqref{N2}.
\end{proof}

Next we calculate $(\Box +1)v_{\mathrm{ap}}$. 

\begin{lemma}\label{lem:vn} 
Assume that $(\phi_{0},\phi_{1})\in Y$. 
Let $u_{\mathrm{ap}}$ be defined by (\ref{up}), where 
$P_{1}, Q_{1}$ and $\Psi$ are given by (\ref{p1}), (\ref{q1}) and (\ref{phase}), 
respectively. Let $v_{\mathrm{ap}}$ be defined by (\ref{vp}), where $P_{n}$ and 
$Q_{n}$ are given by (\ref{pn}) and (\ref{qn}). 
Furthermore, let $N_{\mathrm{nr}}(u_{\mathrm{ap}})$ be given by (\ref{y01}). 
Then, we have 
\begin{eqnarray}
\lefteqn{\|(\Box+1)v_{\mathrm{ap}}-N_{\mathrm{nr}}(u_{\mathrm{ap}})
\|_{L_{x}^{2}}}\label{N3}\\
&\le&Ct^{-2}(\log t)^{2}
\|(\phi_{0},\phi_{1})\|_{Y}^{2}(1+\|(\phi_{0},\phi_{1})\|_{Y}^{2}).
\nonumber
\end{eqnarray}
\end{lemma}

\begin{proof}[Proof of Lemma \ref{lem:vn}.]
In a similar way as in the proof of Lemma \ref{lem:pf1}, we have 
\begin{eqnarray}
\label{h11}\\
(\Box+1)v_{\mathrm{ap}}&=&
t^{-2}\sum_{n=2}^{\infty}(1-n^{2})P_{n}(\mu)
\cos\left(n\langle\mu\rangle^{-1}t
+n\Psi(\mu)\log t\right)\nonumber\\
& &+
t^{-2}\sum_{n=2}^{\infty}(1-n^{2})Q_{n}(\mu)
\sin\left(n\langle\mu\rangle^{-1}t+n\Psi(\mu)\log t\right),
\nonumber\\
& &+\sum_{n=2}^{\infty}R_{n}(t,\mu),
\nonumber
\end{eqnarray}
where $R_{n}$ ($n\ge2$) satisfy the inequalities 
\begin{eqnarray}
\lefteqn{\quad|R_{n}(t,\mu)|}\label{h3}\\
&\le&Ct^{-3}(\log t)^{2}\nonumber\\
& &\times
\sum_{Z=P,Q}\bigl\{\left(n^{2}|Z_{n}(\mu)||\Psi(\mu)|^{2}+n|Z_{n}(\mu)||\Psi(\mu)|
+|Z_{n}(\mu)|\right)
\nonumber\\
& &\qquad\qquad+n\langle\mu\rangle|Z_{n}(\mu)|
\nonumber\\
& &\qquad\qquad+\langle\mu\rangle^{3}
\left(n^{2}|Z_{n}(\mu)||\Psi(\mu)||D\Psi(\mu)|+n|DZ_{n}(\mu)||\Psi(\mu)|\right.
\nonumber\\
& &\qquad\qquad\qquad\ \ \ \left.+n|Z_{n}(\mu)||D\Psi(\mu)|+
|DZ_{n}(\mu)|\right)
\nonumber\\
& &\qquad\qquad+\langle\mu\rangle^{4}
\left(n^{2}|Z_{n}(\mu)||D\Psi(\mu)|^{2}+n|DZ_{n}(\mu)||D\Psi(\mu)|
\right.
\nonumber\\
& &\qquad\qquad\qquad\ \ \ \left.+n|Z_{n}(\mu)||D^{2}\Psi(\mu)|+
|D^{2}Z_{n}(\mu)|\right)\bigl\}.\nonumber
\end{eqnarray}
Then by (\ref{y01}) and (\ref{h11}), we find
\begin{eqnarray}
\|(\Box+1)v_{\mathrm{ap}}-N_{\mathrm{nr}}(u_{\mathrm{ap}})\|_{L_{x}^{2}}
\le \sum_{n=2}^{\infty}\|R_{n}(t)\|_{L_{x}^{2}}.
\label{b2}
\end{eqnarray}
Differentiating (\ref{y0}), we see that $\beta$ satisfies
\begin{eqnarray}
\qquad\quad|D\beta(\mu)|&\le&C\frac{|DP_{1}(\mu)|+|DQ_{1}(\mu)|
}{\sqrt{P_{1}(\mu)^{2}+Q_{1}(\mu)^{2}}},
\label{y1}\\
|D^{2}\beta(\mu)|&\le&C\frac{|D^{2}P_{1}(\mu)|+|D^{2}Q_{1}(\mu)|
}{\sqrt{P_{1}(\mu)^{2}+Q_{1}(\mu)^{2}}}
+C\frac{|DP_{1}(\mu)|^{2}+|DQ_{1}(\mu)|^{2}}{P_{1}(\mu)^{2}+Q_{1}(\mu)^{2}}.
\label{y2}
\end{eqnarray} 
Plugging (\ref{p10})-(\ref{p12}) and (\ref{y1})-(\ref{y2}) into 
(\ref{pn}) and (\ref{qn}), we have
\begin{eqnarray}
\qquad\qquad|Z_{n}(\mu)|
&\le&Cn^{-5}\left(\langle\mu\rangle^{4}|\hat{\phi}_{0}(\mu)|^{2}
+\langle\mu\rangle^{2}|\hat{\phi}_{1}(\mu)|^{2}\right),
\label{y3}\\
|DZ_{n}(\mu)|
&\le&Cn^{-4}|\hat{\phi}_{0}(\mu)|\left(\langle\mu\rangle^{3}|\hat{\phi}_{0}(\mu)|
+\langle\mu\rangle^{4}|D\hat{\phi}_{0}(\mu)|\right)
\label{y4}\\
& &+Cn^{-4}\langle\mu\rangle^{-1}|\hat{\phi}_{1}(\mu)|
\left(\langle\mu\rangle^{2}|\hat{\phi}_{1}(\mu)|
+\langle\mu\rangle^{3}|D\hat{\phi}_{1}(\mu)|\right),
\nonumber
\end{eqnarray}
\begin{eqnarray}
\label{y5}\\
\lefteqn{|D^{2}Z_{n}(\mu)|}\nonumber\\
&\le&Cn^{-3}\biggl\{\left(|\hat{\phi}_{0}(\mu)|+\langle\mu\rangle|D\hat{\phi}_{0}(\mu)|\right)
\left(\langle\mu\rangle^{2}|\hat{\phi}_{0}(\mu)|+\langle\mu\rangle^{3}|D\hat{\phi}_{0}(\mu)|\right)
\nonumber\\
& &\qquad\qquad+|\hat{\phi}_{0}(\mu)|\langle\mu\rangle^{4}|D^{2}\hat{\phi}_{0}(\mu)|\biggl\}
\nonumber\\
& &+Cn^{-3}\biggl\{\left(\langle\mu\rangle^{-1}|\hat{\phi}_{1}(\mu)|
+|D\hat{\phi}_{1}(\mu)|\right)
\left(\langle\mu\rangle|\hat{\phi}_{1}(\mu)|
+\langle\mu\rangle^{2}|D\hat{\phi}_{1}(\mu)|
\right)\nonumber\\
& &\qquad\qquad+\langle\mu\rangle^{-1}|\hat{\phi}_{1}(\mu)|
\langle\mu\rangle^{3}|D^{2}\hat{\phi}_{1}(\mu)|
\biggl\}
\nonumber
\end{eqnarray}
for $Z=P,Q$. 
Substituting (\ref{psi0})-(\ref{psi2}) and (\ref{y3})-(\ref{y5}) 
into (\ref{h3}), we obtain 
\begin{eqnarray*}
\lefteqn{|R_{n}(t,\mu)|}\\
&\le&Cn^{-3}t^{-3}(\log t)^{2}\langle\mu\rangle^{2}\\
& &\times\biggl\{\left(1+\langle\mu\rangle|\hat{\phi}_{0}(\mu)|
\right)^{2}
\left(\langle\mu\rangle^{2}|\hat{\phi}_{0}(\mu)|
+\langle\mu\rangle^{3}|D\hat{\phi}_{0}(\mu)|
\right)^{2}\\
& &\qquad+(1+\langle\mu\rangle|\hat{\phi}_{0}(\mu)|)
\langle\mu\rangle^{2}
|\hat{\phi}_{0}(\mu)|
\langle\mu\rangle^{4}|D^{2}\hat{\phi}_{0}(\mu)|\\
& &\qquad+\left(1+|\hat{\phi}_{1}(\mu)|
\right)^{2}
\left(\langle\mu\rangle|\hat{\phi}_{1}(\mu)|
+\langle\mu\rangle^{2}|D\hat{\phi}_{1}(\mu)|
\right)^{2}\\
& &\qquad+(1+|\hat{\phi}_{1}(\mu)|)
\langle\mu\rangle
|\hat{\phi}_{1}(\mu)|
\langle\mu\rangle^{3}|D^{2}\hat{\phi}_{1}(\mu)|
\biggl\}.
\end{eqnarray*}
Therefore taking $L^{2}$ norm for $R_{n}$ with respect to $x$ 
variable, we have
\begin{eqnarray}
\|R_{n}(t)\|_{L_{x}^{2}}
\le Cn^{-3}t^{-2}(\log t)^{2}
\|(\phi_{0},\phi_{1})\|_{Y}^{2}(1+\|(\phi_{0},\phi_{1})\|_{Y}^{2}).
\label{b3}
\end{eqnarray}
The inequalities (\ref{b2}) and (\ref{b3}) yield 
\begin{eqnarray*}
\lefteqn{\|(\Box+1)v_{\mathrm{ap}}-N_{\mathrm{nr}}(u_{\mathrm{ap}})
\|_{L_{x}^{2}}}\\
&\le&Ct^{-2}(\log t)^{2}\sum_{n=2}^{\infty}n^{-3}
\|(\phi_{0},\phi_{1})\|_{Y}^{2}(1+\|(\phi_{0},\phi_{1})\|_{Y}^{2})
\\
&\le&Ct^{-2}(\log t)^{2}
\|(\phi_{0},\phi_{1})\|_{Y}^{2}(1+\|(\phi_{0},\phi_{1})\|_{Y}^{2}).
\end{eqnarray*}
This completes the proof of Lemma \ref{lem:vn}.
\end{proof}

\begin{proof}[Proof of Proposition \ref{prop:ua}] 
The inequality \eqref{eq:ua1} follows from the definition
of $A$ immediately. To show (\ref{eq:ua2}), 
we first confirm that addition by $v_{\mathrm{ap}}$ 
does not change the main part of the nonlinear part.
However, it is obvious because $v_{\mathrm{ap}}$ 
decays faster than $u_{\mathrm{ap}}$ in time.
Indeed, it can be observed by the elementary inequality
\begin{eqnarray}
\lefteqn{\|N(u_{\mathrm{ap}}+v_{\mathrm{ap}})
-N(u_{\mathrm{ap}})\|_{L_{x}^{2}}}\label{N4}\\
&\le&C(\|u_{\mathrm{ap}}\|_{L_{x}^{\infty}}
+\|v_{\mathrm{ap}}\|_{L_{x}^{\infty}})
\|v_{\mathrm{ap}}\|_{L_{x}^2}\nonumber\\
&\le& Ct^{-2}\|(\phi_{0},\phi_{1})\|_{Y}^{3}(1+\|(\phi_{0},\phi_{1})\|_{Y}).
\nonumber
\end{eqnarray}
By Lemma \ref{lem:pf1} (\ref{N2}),  
Lemma \ref{lem:vn} (\ref{N3}) and (\ref{N4}), 
we see 
\begin{eqnarray*}
\lefteqn{\|(\Box+1)A(t) -N(A(t))\|_{L^2_x}}\\
&\le&\|(\Box+1)u_{\mathrm{ap}}
-N_{\mathrm{r}}(u_{\mathrm{ap}})\|_{L_{x}^{2}}
+\|(\Box+1)v_{\mathrm{ap}}-N_{\mathrm{nr}}(u_{\mathrm{ap}})\|_{L_{x}^{2}}
\\
& &
+\|N(u_{\mathrm{ap}}+v_{\mathrm{ap}})-N(u_{\mathrm{ap}})\|_{L_{x}^{2}}\\
&\le&Ct^{-2}(\log t)^{2}
\|(\phi_{0},\phi_{1})\|_{Y}(1+\|(\phi_{0},\phi_{1})\|_{Y}^{3}).
\end{eqnarray*} 
Hence, we have the inequality \eqref{eq:ua2}. This completes they proof of 
Proposition \ref{prop:ua}. 
\end{proof}

\begin{proof}[Proof of Theorem \ref{thm:main}]
By Proposition \ref{prop:ua}, 
we can apply Proposition \ref{prop:FVP} for 
$A=u_{\mathrm{ap}}+v_{\mathrm{ap}}$. Then there exits a
solution $u$ to (\ref{K}) satisfying (\ref{pr}). Hence 
\begin{eqnarray*} 
\|u-u_{\mathrm{ap}}\|_{L_{x}^{2}}
&\le&\|u-u_{\mathrm{ap}}-v_{\mathrm{ap}}\|_{L_{x}^{2}}
+\|v_{\mathrm{ap}}\|_{L_{x}^{2}}\\
&\le& Ct^{-d}+Ct^{-1}\\
&\le& Ct^{-d},
\end{eqnarray*} 
where $1/2<d<1$. This completes the proof of Theorem \ref{thm:main}.
\end{proof}

\vskip3mm
\noindent {\bf Acknowledgments.} 
The authors would like to thank Professor Hideaki Sunagawa for 
drawing their attention to his related works. 
S.M. is partially supported by the Sumitomo Foundation, Basic Science Research
Projects No.\ 161145. 
J.S. is partially supported by JSPS,
Grant-in-Aid for Young Scientists (A) 25707004.

\end{document}